\documentclass[12pt,a4paper]{amsart}
\usepackage{graphicx}
\usepackage{pgf,tikz,pgfplots}
\usetikzlibrary{arrows}
\usepackage{comment}

\makeatletter
\long\def\unmarkedfootnote#1{{\long\def\@makefntext##1{##1}\footnotetext{#1}}}
\makeatother

\usepackage{amsmath,amssymb}
\usepackage{enumerate,amssymb, hyperref}
\usepackage{color}
\theoremstyle{plain}
\newtheorem{thm}{Theorem}[section]

\newtheorem{lemma}[thm]{Lemma}

\newtoks\prt

\theoremstyle{definition}

\newtheorem{remark}[thm]{Remark}

\newtheorem{definition}[thm]{Definition}

\def\eqn#1$$#2$${\begin{equation}\label#1#2\end{equation}}

\numberwithin{equation}{section}

\def\V{\mathbb V}

\def\epsilon{\varepsilon}
\def\en{\mathbb N}
\def\er{\mathbb R}

\def\H{\mathcal{H}}
\def\L{\mathcal{L}}

\def\loc{\operatorname{loc}}

\def\mir1{\mathcal L_1}

\def\oint{-\hskip -11pt \int}

\def\rn{\mathbb R^n}

\makeatletter
\newcommand{\labeltext}[2]{%
	\@bsphack
	\def\@currentlabel{#1}{\label{#2}}%
	\@esphack
}
\makeatother

\newtoks\by
\newtoks\paper
\newtoks\book
\newtoks\jour
\newtoks\yr
\newtoks\pages
\newtoks\vol
\newtoks\publ
\def\ota{{\hbox\vol{???}}}
\def\cLear{\by=\ota\paper=\ota\book=\ota\jour=\ota\yr=\ota
\pages=\ota\vol=\ota\publ=\ota}
\def\endpaper{\the\by, {\the\paper},
\textit{\the\jour} \textbf{\the\vol} (\the\yr), \the\pages.\cLear}
\def\endbook{\the\by, \textit{\the\book}, \the\publ.\cLear}
\def\endprep{\the\by, \textit{\the\paper}, \the\jour.\cLear}
\def\endyearprep{\the\by, \textit{\the\paper}, \the\jour, (\the\yr).\cLear}
\def\name#1#2{#2 #1}

\definecolor{ffqqqq}{rgb}{1.,0.,0.}
\definecolor{qqffqq}{rgb}{0.,1.,0.}
\definecolor{qqqqff}{rgb}{0.,0.,1.}
\definecolor{xfqqff}{rgb}{0.4980392156862745,0.,1.}
\definecolor{ffxfqq}{rgb}{1.,0.4980392156862745,0.}

\title{A note on Sobolev-Lorentz Capacity and Hausdorff measure}

\author[D. Campbell]{Daniel Campbell}
\address{Department of Mathematical Analysis, Charles University, So\-ko\-lovsk\'a 83, 186~00 Prague 8, Czech Republic}
\email{\tt campbell@karlin.mff.cuni.cz}



\thanks{The author was supported by the grant GA\v{C}R P201/24-10505S and Ministry of Education, Youth and Sport of
	the Czech Republic grant number LL2105 CONTACT}

\date{\today}

\begin{document}

\begin{abstract}
    {In this paper, we give an elementary proof that sets of zero $p,1$-Sobolev-Lorentz capacity are $\H^{n-p}$-null sets, independently of nonlinear potential theory. We further show that there exists a set of Sobolev-Lorentz-$(p,1)$ capacity equal to zero with Hausdorff dimension equal $n-p$.}
\end{abstract}

\maketitle

\section{Introduction}
The study of sets of zero capacity is key in understanding possible pathological behavior of Sobolev functions. It is well known that points have zero Sobolev $n$-capacity, a fact that is closely related to the existence of discontinuous functions in $W^{1,n}_{\loc}(\rn)$. One should compare this to the well-known Morrey estimates for $p>n$. The question arises naturally also in lower dimension, for example when we are interested in the set on which boundary values are satisfied in a pointwise sense (see e.g., \cite{MZ} and \cite{CP}).

The case $p=n$ is particularly interesting, since this is the threshold of regularity required to obtain continuity. Although continuity fails for maps in $W^{1,n}_{\loc}(\rn)$, it was proved in \cite{S} that maps whose derivative belongs to the Lorentz space $L^{n,1}_{\loc}(\rn)$ (for the definition see the Preliminaries), are continuous and differentiable almost everywhere. This is closely connected to the fact that $n,1$-Sobolev-Lorentz null capacity sets are empty (see \cite{SC}). In \cite{KKM}, it was proven that the condition of having derivative in $L^{n,1}_{\loc}(\rn)$ is sharp and it also implies that a mapping satisfies the Lusin (N) condition. It is known that this condition fails for mappings only in $W^{1,n}$, we refer the reader to the example in \cite{MM}.

One is naturally led to ask whether there exists a better qualitative estimate for the size of sets of null $p,1$-Sobolev-Lorentz capacity than in the ordinary Sobolev case, also for other $p$ than $n$. The definition of Sobolev-Lorentz spaces $W^1_{p,q}(\Omega)$ can be found at Definition~\ref{SoboLore} and the (quasi)norm $\|f\|_{L^{p,q}(\Omega)}$ is defined in Definition~\ref{PanLorenc}. We define Sobolev-Lorentz capacities as follows.
\begin{definition}
	Let $n\in \en$, $1<p\leq n$,  $1\leq q\leq \infty$ and let $E\subset \rn$. We define the variational $p,q$-Sobolev-Lorentz capacity of $E$ as 
	$$
	\gamma_{p,q}(E):= \inf\Big\{ \|Df\|_{L^{p,q}(\rn)}^p; f\in W^{1}_{p,q}(\rn) \text{ s.t. } \exists G\supset E \text{ open and } f\geq 1 \text{ on } G     \Big\},
	$$
	we define the relative variational $p,q$-Sobolev-Lorentz capacity of $E$ with respect to an open $\Omega\subset \rn$ as
	$$
	\begin{aligned}
		\gamma_{p,q}(E,\Omega):= \inf\Big\{ \|Df\|_{L^{p,q}(\rn)}^p; &f\in W^{1}_{p,q}(\rn) \text{ s.t. } \exists G\supset E \text{ open and } f\geq 1 \text{ on } G,\\
		& f=0 \text{ on } \rn\setminus \Omega     \Big\},
	\end{aligned}
	$$
	and the $p,q$-Sobolev-Lorentz capacity of $E$ as
	$$
		\gamma_{p,q}^+(E):= \inf\Big\{\|f\|_{L^{p,q}}^p +  \|Df\|_{L^{p,q}(\rn)}^p; f\in W^{1}_{p,q}(\rn) \text{ s.t. } \exists G\supset E \text{ open and } f\geq 1 \text{ on } G     \Big\}.
	$$
	Finally we define $\gamma_{p,q}^+(E,\Omega)$ as the obvious combination of the last two definitions.
\end{definition}

In Lemma~\ref{KapNulovost}, we show that if a set $E\subset \rn$ has zero capacity in any of these capacities then it has zero capacity in all of them.

It was proven in \cite{SC}  that $n,1$-Sobolev-Lorentz null-capacity sets are empty (equivalently, have zero $\H^{0}$ measure). An extension of this to higher derivatives is the sharp Morse-Sard theorem in the Sobolev-Lorentz setting of \cite{KKK}. A further result of \cite{KKK} is a smallness result on the $n-p$ Hausdorff content using the $p,1$-Sobolev-Lorentz capacity. We refer the reader to \cite{MPS} for further results and exposition. Let us also refer the reader to the paper in preparation \cite{KaK}, which gives more detail on $\gamma_{p,1}$-null sets. The aim of this paper, stated in Theorem~\ref{main}, is to give an elementary alternative proof of the smallness result which avoids potential theory.

\begin{thm}\label{main}
	Let $p\in(1,n]$ and let $E\subset \rn$ be such that $\gamma_{p,1}(E) = 0$. Then $\H^{n-p}(E) = 0$. 
\end{thm}

The key to our simplification is to take a detour through Orlicz-Sobolev spaces, utilizing results of Mal\'y, Swanson and Ziemer in \cite{MSZ}. We hope that this more elementary approach will help make the result more accessible and better known, as well as suggesting other possible paths for further inquest. Further motivation for this paper comes from a forthcoming work on the size of graphs of maps in $W^{1,p}_{\loc}(\rn,\er^m)$; see \cite{CMRS}.

Furthermore, we show by example that such a claim does not hold for any $q>1$ (see Theorem~\ref{OneOverK}) and that there exists a set $E$ with $\gamma_{p,1}(E) = 0$ and $\dim_{\H}(E) = n-p$ (see Theorem~\ref{OneOverKSq}).

\section{Preliminaries}
Let $f$ be a measurable function. We define $\mu_{f},$ the \textit{distribution function}
of $f$ as follows (see \cite[Definition II.1.1]{BS}):
$$
	\mu_{f}(t)=\mathcal{L}^n(\{x \in \Omega: |f(x)| > t \}), \qquad t \geq 0.
$$
As can be found in \cite[Definition II.1.5]{BS}, we define $f^{*},$  the \emph{decreasing rearrangement} of $f$ by
$$
	f^{*}(t)=\inf\{v: \mu_{f}(v) \le t \}, \quad t \ge 0.
$$
\begin{definition}\label{PanLorenc}
	Let $n\in \en$, $1<p<\infty$,  $1\leq q\leq \infty$ and let $\Omega\subset \rn$ be measurable. By the $p,q$-Lorentz space on $\Omega$, $L^{p,q}(\Omega)$, we denote the collection of elements of $L^{1}_{\loc}(\Omega)$, such that
	$$
		\|f\|_{L^{p,q}(\Omega)}:=\left\{ \begin{array}{lc}
			\left( \displaystyle{\int_{0}^{\infty} (t^{1/p}f^{*}(t))^q \, \frac{dt}{t}}
			\right)^{1/q} & 1 \le q < \infty \\
			\\
			\sup\limits_{t>0} t \mu_{f}(t)^{1/p}=\sup\limits_{s>0}
			s^{1/p} f^{*}(s) & q=\infty
		\end{array}\right.
	$$
	is finite. We often simplify the notation $\|f\|_{p,q} := \|f\|_{L^{p,q}(\Omega)}$, where there is no danger of misinterpretation.
\end{definition}

\begin{remark}
	In the case that $q>p$ it is common to replace $f^*$ with $f^{**}$ in the above definition so that we have a norm, where
	$$
		f^{**}(t)=\mu_{f^{*}}(t)=\frac{1}{t} \int_{0}^{t} f^{*}(s) ds, \quad t >0.
	$$
	The two values are comparable and we do not need to do this in our paper. Moreover, we are especially interested in small values of $q$ where $\|\cdot\|_{L^{p,q}(\Omega)}$ is a norm as the other cases can be retrieved by an embedding argument. For the details we refer the reader to \cite{PKJF}, especially to Section~8.2.
	
	Further, it can be proven that
	$$
		\|f\|_{L^{p,1}(\Omega)}= \int_{0}^{\infty} t^{1/p}f^{*}(t) \, \frac{dt}{t}  = \int_{0}^{\infty} \big(\mu_f(s)\big)^{\frac{1}{p}}	\, ds.
	$$
	The proof is easy for simple functions and in general holds by approximation. We refer the reader to \cite{G} for details.
\end{remark}

\begin{definition}\label{SoboLore}
	Let $n\in \en$, $1<p<\infty$,  $1\leq q\leq \infty$ and let $\Omega\subset \rn$ be open. By the $p,q$-Sobolev-Lorentz space on $\Omega$, $W^1_{p,q}(\Omega)$, we denote the collection of elements of $W^{1,1}_{\loc}(\Omega)$, such that
	$$
		\|f\|_{1,p,q} :=\|f\|_{p,q} +\|Df\|_{p,q}<\infty.
	$$
\end{definition}

We take the following definition from \cite{PKJF} (especially see Definition~4.2.1, Definition~4.8.1 and  Theorem~4.8.5).
\begin{definition}
	A function $\Phi :[0,\infty) \to [0,\infty)$ is said to be a Young function if there exists a right-continuous, non-decreasing $\varphi:(0,\infty)\to (0,\infty)$ satisfying
	$$
	\lim_{t\to 0^+}\varphi(t)=  0, \quad \lim_{t\to \infty}\varphi(t)=  \infty
	$$
	such that 
	$$
		\Phi(t) = \int_0^t \varphi(s) \, ds.
	$$
	Let $\Omega\subset \rn$ be measurable. The Orlicz space $L^{\Phi}(\Omega)$ is defined as the space of all measurable functions $g$
	on $\Omega$ for which there exists a $\lambda >0$ such that
	$$
		\int_{\Omega} \Phi\Big(\frac{|g|}{\lambda}\Big)\, d\L^n \leq 1
	$$
	and $\|g\|_{L^{\Phi}(\Omega)}$ is the infimum of all such $\lambda$.
\end{definition}

The following is due to \cite[Lemma~9.3]{MSZ}.
\begin{thm}\label{MSZthm2}
	Let $p\in(1,\infty)$. Suppose that $f\in L^{p,1}(\Omega)$. Then there exists a Young function $\Phi$ satisfying
	$$
		\int_{0}^{\infty}(\Phi'(t))^{-\frac{1}{p-1}} \, dt =1
	$$
	and
	$$
		\int_{\Omega}\Phi(|f|)\, d\L^n \leq 1.
	$$
\end{thm}

The following is \cite[Theorem~8.13]{MSZ}.
\begin{thm}\label{MSZthm3}
	Suppose that $p > 1$ and $\Phi$ is a Young function satisfying
$$
\int_0^{\infty} \big(\Phi'(t)\big)^{-\frac{1}{p-1}} \, dt \leq 1.
$$
Let $E\subset\rn$. There is a $C>0$ depending only on $n$ and $p$ such that whenever $f \in W^{1,1}_{\loc}(\rn)$ is precisely represented and $f \geq 1$ on $E$ it holds that 
$$
\H^{n-p}(E) \leq C \int_{\rn} (\Phi(|f|) + \Phi(|Df|)) \, d\L^n.
$$
\end{thm}

\section{Proof of Theorem~\ref{main}}

Before we start the proof of Theorem~\ref{main} itself, let us prove the following preparatory lemma. The lemma also holds for the case $p=n$ with a slightly altered proof.
\begin{lemma}\label{KapNulovost}
	Let $1<p<n$, let $G\subset \rn$ be open. The following are equivalent
	$$
	\begin{aligned}
		\text{a) }&\gamma^+_{p,1}(E) = 0\\
		\text{b) }&\gamma^+_{p,1}(E,G) = 0\\
		\text{c) }&\gamma_{p,1}(E) = 0\\
		\text{d) }&\gamma_{p,1}(E,G) = 0.
	\end{aligned}
	$$
\end{lemma}
\begin{proof}
	It is obvious that
	$$
	\gamma_{p,1}(E) \leq \gamma_{p,1}^+(E)\leq \gamma_{p,1}^+(E,G)
	$$
	and
	$$
	\gamma_{p,1}(E)\leq \gamma_{p,1}(E,G)\leq \gamma_{p,1}^+(E,G).
	$$

	Without loss of generality we may assume that $E\neq \emptyset$. Let us assume that $\gamma_{p,1}(E)=0$. To begin with let us assume that $E\Subset G$. By the subadditivity of the capacity (for example see \cite[Theorem~5.1]{SC}) we may assume that $E\Subset G \cap  B(0,1)$ and it suffices to show the claim for $G$ of finite measure. We find an open set $V \Subset G$, and for each $\epsilon >0$, we find a $v\in W^{1,1}_{\loc}(\rn)$ such that $v\geq 1$ on an open set $U_{\epsilon}$, $E\Subset U_{\epsilon} \subset V \Subset G$ and 
	$$
	\|Dv\|_{L^{p,1}(\rn)}^p\leq \epsilon.
	$$
	Further we find an $\eta\in \mathcal{C}^\infty_c(\rn)$ such that $\eta=1$ on $V$, $\eta(\rn)= [0,1]$,  $\eta = 0$ outside $G$ and $|\nabla\eta|\leq M$ for some $M>1$. We define $\tilde{v} = \eta v$, then $\tilde{v}$ is a contender in the infimum for $\gamma_{p,1}^+(E,G)$. Using the embedding of $L^{p,1}(G)$ into $L^{p}(G)$ and the embedding of $L^{\frac{np}{n-p}}(G)$ into $L^{p,1}(G)$ (see \cite[Chapter~8.2]{PKJF}) and the Sobolev inequality on $\rn$, we have
	$$
	\begin{aligned}
		\gamma_{p,1}^+(E,G) &\leq \|\tilde{v}\|^p_{p,1} + \|D\tilde{v}\|_{p,1}^p\\
		&\leq CM^p \|\chi_Gv\|_{p,1}^p + C\|Dv\|_{p,1}^p\\
		&\leq CM^p \|\chi_Gv\|_{\frac{np}{n-p}}^p + C\|Dv\|_{p,1}^p\\
		& \leq CM^p\|Dv\|_p^p + C \|Dv\|_{p,1}^p\\
		& \leq CM^p\|Dv\|_{p,1}^p\\
		&\leq CM^p \epsilon.
	\end{aligned}
	$$
	We send $\epsilon$ to zero prove that $\gamma_{p,1}^+(E,G) = 0$.
	
	To conclude the proof it suffices to write $E = \bigcup_{i=1}^{\infty}E_i$, where each $E_i\Subset G$ and use the subadditivity of the capacity.
\end{proof}

Now we may proceed to the proof of Theorem~\ref{main}.
\begin{proof}[Proof of Theorem~\ref{main}]
	The proof in the case $p=n$ can be found in \cite{SC}, therefore we assume that $p<n$.
	
	Let $E$ be a set such that $\gamma_{p,1}(E) = 0$. Then, by Lemma~\ref{KapNulovost}, it holds that $\gamma_{p,1}^+(E,G) = 0$ for any arbitrary open $G$ containing $E$. We define a sequence of functions $f_k \in W^{1,1}_{\loc}(\rn)$ inductively. We find $f_1 =1$ on $G_1 \supset E$ and $\|f_1\|_{p,1}^p + \|Df_1\|_{p,1}^p \leq 2^{-1}$. Having found $f_{k-1}=1$ on $G_{k-1}\supset E$ with $\|f_{k-1}\|_{p,1}^p + \|Df_{k-1}\|^p_{p,1}\leq 2^{1-k}$ we find an $f_{k}$ and a $E\subset G_k\Subset G_{k-1}$ and (since $\gamma_{p,1}(E,G_{k-1}) = 0$) we may assume that $\|f_k\|_{p,1}^p +\|Df_k\|_{p,1}^p\leq 2^{-k}$. By the absolute continuity of the norm on $L^{p,1}$ we have the density of smooth functions thanks to the standard convolution approximation arguments. Therefore we may assume that $f_k$ are smooth maps. We define $F:=\sum_{k=1}^{\infty}|f_k|+|Df_k|$. It holds that $F\in L^{p,1}(\rn)$.
	
	By Theorem~\ref{MSZthm2} we find a Young function $\Phi$ such that
	\begin{equation}\label{FinitenessOfExistence}
			\int_{0}^{\infty}\big[\Phi'(t)\big]^{-\frac{1}{p-1}} \, dt =1
	\end{equation}
	and
	$$
		\int_{\rn}\Phi(F)\, d\L^n <\infty.
	$$
	We define
	$$
		g_k(x) = \begin{cases}
			\sum_{j=k}^{\infty}f_j(x),\quad &x\in \rn\setminus E\\
			\infty &x\in E .
		\end{cases}
	$$
	Then each $g_k$ is smooth on each of the sets $\rn\setminus G_m$, $m\in \en$. Further, it can easily be observed that $\lim_{r\to 0}\oint_{B(x,r)}g_k = \infty$ for all $x\in E$. Thus we see that $g_k$ is its own precise representative. Since for almost every $x$ there is at most one $k$ such that $Df_k\neq 0$ we have that 
	$$
		Dg_k(x) = \sum_{j=k}^{\infty}Df_j(x).
	$$
	It follows from Theorem~\ref{MSZthm3} and the dominated convergence theorem that
	$$
		\H^{n-p}(E) =\lim_{k\to\infty}\int_{\rn}\Phi(|Dg_k|) + \Phi(g_k)=0.
	$$
\end{proof}

\section{Examples showing the sharpness of the result}
Throughout this chapter we denote cubes by the letter $Q$. Specifically, let $x\in \rn$ and $r>0$, we define the set
$$
	Q(x,r) := \prod_{i=1}^n(x_i-r, x_i+r).
$$
\begin{thm}\label{OneOverK}
	Let $n\in \en$, $1<p< n$ and $1<q\leq \infty$. There exists a set $E\subset \rn$ with $\H^{n-p}(E)>0$ and a sequence of continuous $f_j\in W^{1,1}_{\loc}(\er^n)$ with
	$$
		\|f_j\|_{L^{p,q}(\rn)} + \|Df_j\|_{L^{p,q}(\rn)} \to 0
	$$
	as $j\to \infty$ and $f_j \geq 1$ on open sets $G_j \supset E$.
\end{thm}
\begin{proof}
	For a convenient shorthand we define
	$$
	\V_n := \{-1,1\}^n\quad \text{ and }  \quad  \beta:= \frac{p}{n-p}.
	$$
	Let $k\in \en$ and let $w=(w_1,w_2,\dots,w_k)\in (\V_n)^k$, we define
	\begin{equation}\label{Pretty1}
		c_w:=\sum_{i=1}^{k} 2^{\beta-i(\beta+1)}w_i
		\quad\text{and}\quad
		A_w:=Q\Big(c_w, 2^{\beta-k(\beta +1)}\Big)\setminus Q\Big(c_w, 2^{-k(\beta +1)}\Big).
	\end{equation}
	The reader may refer to Figure~\ref{Fig:CantorConstruction} for a depiction of the points $c_w$ and the sets $A_w$ in dimension one. This is a standard Cantor-type set. The set $A_w$ is a cubic annulus whose center is $c_w$. The union over $w\in \V^k$ of $A_w$ is often referred to as the `frame of generation $k$'.
	 
	 \definecolor{wewdxt}{rgb}{0.2,0.2,0.2}
	 \definecolor{lr}{rgb}{0.9,0.1,0.1}
	 \definecolor{rg}{rgb}{0.1,0.9,0.1}
	 \definecolor{mg}{rgb}{1,0.1,0.6}
	 \definecolor{bg}{rgb}{.1,0.6,1}
	 \definecolor{rvwvcq}{rgb}{0.08235294117647059,0.396078431372549,0.7529411764705882}
	 \begin{figure}
	 	\begin{tikzpicture}[line cap=round,line join=round,>=triangle 45,x=7cm,y=7cm]
	 		\clip(-1.2,-0.5) rectangle (1.2,0.2);
	 		\draw [line width=2pt,color=wewdxt] (-0.75,-0.2)-- (-0.25,-0.2);
	 		\draw [line width=2pt,color=wewdxt] (0.25,-0.2)-- (0.75,-0.2);
	 		\draw [line width=2pt,color=wewdxt] (-0.6875,-0.4)-- (-0.5625,-0.4);
	 		\draw [line width=2pt,color=wewdxt] (-0.4375,-0.4)-- (-0.3125,-0.4);
	 		\draw [line width=2pt,color=wewdxt] (0.3125,-0.4)-- (0.4375,-0.4);
	 		\draw [line width=2pt,color=wewdxt] (0.5625,-0.4)-- (0.6875,-0.4);
	 		\draw [line width=2pt,color=wewdxt] (-1,0)-- (1,0);
	 		\draw [line width=2pt,color=mg] (-1,-0.2)-- (-0.75,-0.2);
	 		\draw [line width=2pt,color=mg] (-0.25,-0.2)-- (-0,-0.2);
	 		\draw [line width=2pt,color=bg] (0.25,-0.2)-- (-0,-0.2);
	 		\draw [line width=2pt,color=bg] (0.75,-0.2)-- (1,-0.2);
	 		
	 		\draw [line width=2pt,color=mg] (-0.75,-0.4)-- (-0.6875,-0.4);
	 		\draw [line width=2pt,color=mg] (-0.5625,-0.4)-- (-0.5,-0.4);
	 		\draw [line width=2pt,color=bg] (-0.5,-0.4)-- (-0.4375,-0.4);
	 		\draw [line width=2pt,color=bg] (-0.3125,-0.4)-- (-0.25,-0.4);
	 		\draw [line width=2pt,color=bg] (0.75,-0.4)-- (0.6875,-0.4);
	 		\draw [line width=2pt,color=bg] (0.5625,-0.4)-- (0.5,-0.4);
	 		\draw [line width=2pt,color=mg] (0.5,-0.4)-- (0.4375,-0.4);
	 		\draw [line width=2pt,color=mg] (0.3125,-0.4)-- (0.25,-0.4);
	 		\begin{scriptsize}
	 			\draw [fill=rvwvcq] (0,0) circle (2.5pt);
	 			\draw[color=rvwvcq] (0.0,0.03) node {$c_0$};
	 			\draw [fill=lr] (-0.5,0) circle (2.5pt);
	 			\draw[color=lr] (-0.5,0.03) node {$c_{-1}$};
	 			\draw [fill=rg] (0.5,0) circle (2.5pt);
	 			\draw[color=rg] (0.5,0.03) node {$c_{1}$};
	 			\draw[color=wewdxt] (0.0,0.08) node {$Q(c_{0}, 2^{0})$};
	 			\draw[color=wewdxt] (-0.5,-0.12) node {$Q(c_{-1}, 2^{-1-1})$};
	 			\draw[color=wewdxt] (0.5,-0.12) node {$Q(c_{1}, 2^{-1-1})$};
	 			\draw[color=mg] (-0.85,-0.12) node {$A_{-1}$};
	 			\draw[color=mg] (-0.15,-0.12) node {$A_{-1}$};
	 			\draw[color=bg] (0.85,-0.12) node {$A_{1}$};
	 			\draw[color=bg] (0.15,-0.12) node {$A_{1}$};
	 			
	 			\draw[color=mg] (-0.625,-0.32) node {$A_{-1,-1}$};
	 			\draw[color=bg] (-0.375,-0.32) node {$A_{-1,1}$};
	 			\draw[color=bg] (0.625,-0.32) node {$A_{1,1}$};
	 			\draw[color=mg] (0.375,-0.32) node {$A_{1,-1}$};
	 			
	 			\draw [fill=lr] (-0.625,-0.2) circle (2.5pt);
	 			\draw[color=lr] (-0.625,-0.17) node {$c_{-1,-1}$};
	 			\draw [fill=rg] (-0.375,-0.2) circle (2.5pt);
	 			\draw[color=rg] (-0.375,-0.17) node {$c_{-1,1}$};
	 			\draw [fill=lr] (0.375,-0.2) circle (2.5pt);
	 			\draw[color=lr] (0.375,-0.17) node {$c_{1,-1}$};
	 			\draw [fill=rg] (0.625,-0.2) circle (2.5pt);
	 			\draw[color=rg] (0.625,-0.17) node {$c_{1,1}$};
	 			\draw [fill=lr] (-0.65625,-0.4) circle (2.5pt);
	 			\draw[color=lr] (-0.65625,-0.37) node {$c_{-1,-1,-1}$};
	 			\draw [fill=rg] (-0.59375,-0.4) circle (2.5pt);
	 			\draw[color=rg] (-0.54,-0.44) node {$c_{-1,-1,1}$};
	 			\draw [fill=lr] (-0.40625,-0.4) circle (2.5pt);
	 			\draw[color=lr] (-0.40625,-0.37) node {$c_{-1,1,-1}$};
	 			\draw [fill=rg] (-0.34375,-0.4) circle (2.5pt);
	 			\draw[color=rg] (-0.29,-0.44) node {$c_{-1,1,1}$};
	 			\draw [fill=lr] (0.34375,-0.4) circle (2.5pt);
	 			\draw[color=lr] (0.34375,-0.37) node {$c_{1,-1,-1}$};
	 			\draw [fill=rg] (0.40625,-0.4) circle (2.5pt);
	 			\draw[color=rg] (0.45,-0.44) node {$c_{1,-1,1}$};
	 			\draw [fill=lr] (0.59375,-0.4) circle (2.5pt);
	 			\draw[color=lr] (0.59375,-0.37) node {$c_{1,1,-1}$};
	 			\draw [fill=rg] (0.65625,-0.4) circle (2.5pt);
	 			\draw[color=rg] (0.7,-0.44) node {$c_{1,1,1}$};
	 		\end{scriptsize}
	 	\end{tikzpicture}
	 	\caption{The first three generations of the construction of a Cantor set in dimension 1 as the limit of finite sums $c_w$ and an illustration of the sets $A_w$ in generation one and two.}\label{Fig:CantorConstruction}
	 \end{figure}

	The set 
	$$
		E:=  \Big\{ \sum_{i=1}^{\infty} 2^{\beta-i(\beta+1)}w_i; w\in \V_n^{\en}\Big\} = \bigcap_{k=1}^{\infty}\bigcup_{w\in \V_n^k} \overline{Q\big(c_w, 2^{-k(\beta +1)}\big)}.
	$$
	is a compact Cantor-type set of positive finite $\H^{n-p}$ measure (for this standard computation see \cite{F}). In order to aid the comprehension of the reader, let us note that in the case that $n=1$ and $\beta= \frac{\ln2 - \ln 3}{\ln 2}$ we retrieve a version of the standard Cantor ternary set. Varying values of $\beta$ give a set with similar properties but of a differing Hausdorff dimension. The set for $n>1$ can be equated with the Cartesian products of $n$ copies of the one-dimensional case.
	
	It holds that
	$$
	Q(0,1) = \bigcup_{k=1}^{\infty}\bigcup_{w\in \V_n^k}A_w \cup E \cup N
	$$
	where the sets in the above union are all pairwise disjoint and the set $N$ has zero Lebesgue measure. That is to say almost every point of $Q(0,1)$ belongs to exactly one frame of generation $k$, or is a point of the Cantor set $E$.
	
	We define $|w|:= k$ if $w\in \V_n^{k}$, and $|w|=\infty$ if $w \in \V_n^{\en}$. For each $j\in \en$ we find the smallest $J = J(j)$ such that $C_j :=\sum_{k=j}^{J}k^{-1} \geq 1$. We define the function $f_j$ as follows
	$$
	f_j(x):=\begin{cases}
		0\quad &x\in A_w \text{ for }1\leq |w| \leq j-1\\
		C_j^{-1}\sum_{k=j}^{|w-1|}\frac{1}{k} + \frac{1}{|w|} \Big(\frac{2^{-(\beta+1)|w|+ \beta} - |x- c_w|_{\infty}}{2^{-(\beta+1)|w|+ \beta} - 2^{-(\beta + 1)|w|}}\Big) \quad & x\in A_w \text{ for }j\leq |w|\leq J\\
		1\quad& x\in \bigcup_{w\in \V_n^{J}}Q\big(c_w, 2^{-J(\beta +1)}\big).\\
	\end{cases}
	$$

	It is easy to verify that $f_j$ is Lipschitz continuous.  Specifically, we estimate the $|Df_j|$ on $\bigcup_{w\in \V^k}A_w$ as 
	$$
	|Df_j(x)|\leq \begin{cases}
		\frac{1}{k}\frac{1}{2^{-(\beta+1)k+ \beta} - 2^{-(\beta + 1)k}} \approx \frac{2^{k(\beta+1)}}{k} \qquad & x\in \bigcup_{w\in \V^k}A_w  \ \& \ j\leq k\leq J\\
		0 \qquad & \text{else.} \\
	\end{cases}
	$$
	By summing the measures of the set $A_w$ over $w\in \V^k$, we have the following measure estimates
	$$
	\L^n\Big(\bigcup_{w\in \V^k}A_w\Big) = 2^{nk}\big(2^{n\beta-nk(\beta +1)} - 2^{-nk(\beta +1)}\big) \approx 2^{-nk\beta}.
	$$
	Thus, we see that the measure of the set where $|Df_j| \geq C2^{k(\beta+1)}k^{-1}$ is bounded by $C2^{-nk\beta}$. Therefore, for $k\geq j$ and $C2^{-(k+1)n\beta}<t<C2^{-kn\beta}$ we have
	$$
		|Df|^*(t) \leq Ck^{-1}2^{k(\beta+1)}.
	$$
	Recalling the fact that $\beta := \frac{p}{n-p}$ we calculate
	$$
		\|Df_j\|_{p,q}^q = \int_{0}^{\infty}[t^{\frac{1}{p}}|Df_j|^*(t)]^q \frac{dt}{t} \leq C\sum_{k=j}^{\infty}\frac{1}{k^q}\xrightarrow{j\to\infty} 0
	$$
	for $1<q<\infty$. The calculation for $q=\infty$ is even simpler, i.e.
	$$
		\|Df_j\|_{L^{p,\infty}} \leq \frac{C}{j}.
	$$
	Since $f_j$ are uniformly bounded by 1 and supported on a set of measure at most $2^{-j\beta}$ we easily observe that $\|f_j\|_{L^{p,q}(\rn)} \to 0$ as $j\to \infty$. It follows that $\gamma_{p,q}(E) = 0$ by definition.
\end{proof}

\begin{thm}\label{OneOverKSq}
	Let $n\in \en$ and $1<p< n$. There exists a set $E\subset \rn$ with $\dim_{\H}(E)=n-p$ (but $\H^{n-p}(E) = 0$) and a sequence of continuous $f_j\in W^{1,1}_{\loc}(\er^n)$ with
	$$
	\|f_j\|_{L^{p,1}(\rn)} + \|Df_j\|_{L^{p,1}(\rn)} \to 0
	$$
	as $j\to \infty$ and $f_j \geq 1$ on open sets $G_j \supset E$.
\end{thm}
\begin{proof}
	The construction is somewhat similar to that in Theorem~\ref{OneOverK}. We use the similar notation
	$$
		\V_n := \{-1,1\}^n, \quad  \beta:= \frac{p}{n-p},
	$$
	for $k\in \en$ and $w\in \V_n^k$ define
	$$
		c_w:=\sum_{i=1}^{k} \frac{2^{\beta-i(\beta+1)}}{i}w_i, \quad A_w:=Q\Big(c_w, \frac{2^{\beta-k(\beta +1)}}{(k-1)}\Big)\setminus Q\Big(c_w, \frac{2^{-k(\beta +1)}}{k}\Big),
	$$
	and
	$$
		 E:=  \Big\{ \sum_{i=1}^{\infty} \frac{2^{\beta-i(\beta+1)}}{i}w_i; w\in \V_n^{\en}\Big\} = \bigcap_{k=1}^{\infty}\bigcup_{w\in \V_n^k} \overline{Q\big(c_w, 2^{-k(\beta +1)k^{-1}}\big)} .
	$$
	The reader may refer to Figure~\ref{Fig:CantorConstruction} for a clarification of the points $c_w$ and the sets $A_w$ in the case $n=1$. Again, we denote by $J = J(j)$  the smallest number such that $C_j :=\sum_{k=j}^{J}k^{-1} \geq 1$. We define the function $f_j$ as follows
	$$
	f_j(x):=\begin{cases}
		0\quad &x\in A_w \text{ for }1\leq |w| \leq j-1\\
		C_j^{-1}\sum_{k=j}^{|w-1|}\frac{1}{k} + \frac{1}{|w|} \Big(\frac{|w|\cdot2^{-(\beta+1)|w|+ \beta} - (|w|-1)|x- c_w|_{\infty}}{|w|\cdot 2^{-(\beta+1)|w|+ \beta} - (|w|-1)2^{-(\beta + 1)|w|}}\Big) \quad & x\in A_w \text{ for }j\leq |w|\leq J\\
		1\quad& x\in \bigcup_{w\in \V_n^{J}}Q\Big(c_w, \frac{2^{-J(\beta +1)}}{J}\Big).\\
	\end{cases}
	$$

	Similarly to before, we have that $f_j$ are Lipschitz. Repeating the calculation from the previous example (using the definition of $\beta = \frac{p}{n-p}$) we see that
	$$
		|Df_j|^*(t) \approx 2^{k\frac{n}{n-p}} \quad \text{for } t\in \Big(C\frac{2^{-(k+1)\frac{np}{n-p}}}{(k+1)^n}, C\frac{2^{-k\frac{np}{n-p}}}{k^n}\Big)
	$$
	and therefore, for $p<n$,
	$$
	\|Df_j\|_{p,1} = \int_{0}^{\infty}t^{\frac{1}{p}}|Df|^*(t) \frac{dt}{t} \leq C\sum_{k=j}^{\infty}\frac{1}{k^{\frac{n}{p}}} \xrightarrow{j\to\infty} 0.
	$$
	
	By covering $E$ with the sets $\bigcup_{w\in \V_n^{k}}Q\Big(c_w, \frac{2^{-k(\beta +1)}}{k}\Big)$, we estimate
	$$
	\H^{n-p}(E) \leq 2^{nk}2^{-k(\beta+1)(n-p)}k^{p-n} = k^{p-n} \to 0
	$$
	as $k\to \infty$. For $d<n-p$ we use the standard argument of the optimality of our covering (see \cite{F}) and the same coverings give
	$$
	\H^{d}_{2^{-k(\beta+1)}k^{-1}}(E) \geq C2^{nk}2^{-k(\beta+1)d}k^{-d} = C2^{kn\frac{n-p-d}{n-p}}k^{-d} \to \infty
	$$
	as $k\to \infty$, proving that $\dim_{\H}E = n-p$.
\end{proof}

\end{document}